%% file: CG-Decomp-public.tex
\documentclass[12pt]{article}
\usepackage{amsmath,amssymb,amsthm}%,mathtools}
\usepackage{fouriernc,graphicx}

\newtheorem{theorem}{Theorem}[section]
\newtheorem{lemma}[theorem]{Lemma}
\newtheorem{prop}[theorem]{Proposition}

\newtheorem{deF}[theorem]{Definition}

\newtheorem{rulE}[theorem]{Rule}

\input{basic-layout}

\input{MathMacro}

\input{algebra}

\input{zec-local}

\begin{document}

 \begin{center}
\bf\Large
The Chung-Graham Expansion\\ 
\rm\normalsize
\VS{1.5em}
Sungkon Chang   
\end{center} 
 
  \subsection*{Abstract}
 
Chung and Graham introduced a method to uniquely represent each positive integer using even-indexed Fibonacci terms. We generalize this result to represent each positive integer using other  Fibonacci terms with equally-spaced indices.

\section{Introduction}  \label{sec:intro}

Let $\seq Q$ be a sequence of real numbers in the interval $[0,1]$.
 In \cite{graham-1984}, 
Chung and Graham considered a measure  of irregularity in the sequence
suggested by a question of D. J. Newman (see \cite{erdos-graham}):
$$  C=\inf_n \ \liminf_{m\to \infty}\  n\, \abs{ Q_{m+n}-Q_m },$$
which is a more refined version of the measure introduced by Bruijn and Erd\"os  in \cite{erdos}. 
Chung and Graham proved that   
$$C\le \al:=\left( 1 + \sum_{k=1}^\infty \frac 1 {F_{2k}} \right)^{-1} \approx 
0.39441967$$
where $F_k$ are the terms of the \fib\ sequence, i.e., 
$(F_1,F_2)=(1,1)$ and $F_{k+2} = F_{k+1} + F_k$ for $k\ge 1$.
What a surprising appearance of the \fib\ sequence in the seemingly unrelated problem!
Moreover, for each positive integer $n$, 
they construct a sequence $\seq {\ep\lv n}$ in $\set{0,1,2}$, and prove that 
if $Q_n= \al \sum_{k=1}^\infty \ep\lv n_k / F_{2k}$, then 
the value of $C$ for $\set{Q_n}_{n=1}^\infty$ is equal to $\al$, which 
implies that $\al$ is the least upper bound of $C$.

In \cite{graham-1984},
Chung and Graham show that 
 for each positive integer $n $,
  there is 
a unique sequence $ \seq {\ep\lv n}  $ 
 such that 
 \begin{equation} 
 n=\sum_{k=1}^\infty  \ep\lv n_k F_{2k} \label{eq:CG-expansion}
\end{equation} 
and 
the terms $\ep\lv n_k$ satisfy Rule \ref{rule:CG} given below.
 They use these sequences $\seq {\ep\lv n}$ to reach the least upper bound of $C$.
 \begin{rulE} \ %
 \rm \label{rule:CG}
\begin{enumerate}
\item
$\ep\lv n_k \in\set{0,1,2}$ for all $k\in\nat$.
\item
If $\ep\lv n_k=\ep\lv n_j=2$ where $1\le k<j$, then 
 $\ep\lv k_i=0$ for some $k<i<j$.
 In particular, $2=\ep\lv n_k=\ep\lv n_{k+1}$ is not allowed.
\end{enumerate} 
\end{rulE}  
\noindent
 We call the summation (\ref{eq:CG-expansion}) {\it the Chung-Graham expansion of $n$},
 and the conditions described in Rule \ref{rule:CG}
 {\it the Chung-Graham rule of expansion}.

 This reminds us of \zec's Theorem,
 which   states that 
every positive integer can be uniquely written as a sum of 
 non-adjacent distinct \fib\ terms.
 In other words, if $n $ is a positive integer, then
 there is a unique sequence $\seq \ep$ 
 such that $n=\sum_{k=1}^\infty \ep_k F_{k}$ and
\begin{equation}
\ep_1=0,\ \ep_k\in\set{0,1},\ 
 \ep_k=1 \Rightarrow \ep_{k+1}=0\ \text{for all }k\in\nat. \label{eq:Z-rule}
\end{equation} 
 The values of  $F_1$ and $F_2$ are equal to each other, and by requiring $\ep_1=0$, 
  we achieve the uniqueness of the sequence $\seq \ep$.
These examples lead us to the broader question of expanding   positive integers by a sequence,
a topic with a rich body of literature; see 
\cite{berthe,brown-61,bruckman,CFHMN2,miller:minimality,daykin,miller:f-dec,fraenkel,frigeri,grabner,hogatt,Lek}.

Let us introduce terminology to  discuss the topic more precisely.
  Let $\nat$ and $\nat_0$ denote the set of positive integers and 
 the set of non-negative integers, respectively.
 We denote a sequence $\seq G$ simply by $G$, and  
 identify the sequence with the infinite tuple consisting of its terms, i.e., $G=(G_1,G_2,G_3,\dots)$.
\begin{deF}
\rm
Let  $ \ep $ be a sequence in $\nat_0$, and 
 let  $ G$ be a sequence in $\nat$.
  If $n=\Cdot \ep G\in\nat_0$, then 
 we call the summation  an {\it expansion of $n$ by  $ G$}  and 
 the sequence $  \ep$ a {\it\cs}.
 By definition, \cs s have only finitely many non-zero entries.
 If there is a set of conditions that $ \ep$ is required to satisfy, we call the set a {\it \roe}.
\end{deF}

 For the Chung-Graham expansions, the \roe\  is as stated 
 in Rule \ref{rule:CG}.
 For the \zec\ expansions, the \roe\ is the conditions listed in (\ref{eq:Z-rule}). 
 
The natural question for us is how to generalize the Chung-Graham expansions using other terms with equally spaced indices.
 For example, if we  use $F_{4k-2}$ for $k\ge 1$, i.e., 
 $$F_2, F_6, F_{10},F_{14},\dots,$$
  what should be a \roe\  under which each positive integer
  can be uniquely expanded by these terms?  
We provide an answer to this question in this note.  In fact, we introduce  a rule of expansion
for the sequence $\set{F_{2+d(k-1)}}_{k=1}^\infty $ where $d\in\nat$ is a fixed even integer.  
 Our approach can be applied when $d$ is odd or when the sequence is $\set{F_{1+d(k-1)}}_{k=1}^\infty $
 for all integers $d\ge 2$.
 The work  is nearly identical to
 the case where $d>0$ is even, which is treated here, with the terms $F_{2+d(k-1)}$, and 
 we leave the details to the reader for simplicity in the presentation of this note.

 \section{The \roe} \label{sec:roe}
Let $d$ be a positive even integer.  We introduce the \roe\ for the sequence $\set{F_{2+d(k-1)}}_{k=1}^\infty $.
Let $K$ be the sequence given by $(K_1,K_2)=(1,3)$ and 
 $K_{k+2}= K_{k+1} + K_k$ for all $k\in\nat$.
 Then, $$K=(1,3,4,7,11,18,29,\dots).$$
 
\begin{deF}
\rm \label{def:rule}
Let $A:=K_d-1$ and $B:=F_{2+d}-1$.  
A \cs\ $\ep$, which has only finitely many non-zero entries by definition, is said to satisfy {\it the Chung-Graham \roe\  for even interval $d$} if 
the following are satisfied:
\begin{enumerate}
\item $0\le \ep_1\le B$ and $0\le \ep_k\le A$ for all $k\ge 2$.
 
\item  
If there is an index $m\ge 2$ such that $\ep_m=A$ and $\ep_k=A-1$ for all $2\le k\le m-1$, then $\ep_1 <B$.

\item If $\ep_k=\ep_j=A$ for $2\le k<j$, then there is an index $c$ such that $k<c<j$ and 
$\ep_c \le A-2$. In particular,  $ \ep_k=\ep_{k+1}=A$ for $k\ge 2$ is not allowed.

\end{enumerate} 
\end{deF}

Let $d=2$. Then, $A=B=2$, and it is equivalent to the \CGR\ introduced in Section 
\ref{sec:intro}. Item (1) states $\ep_k\in\set{0,1,2}$ for all $k\in\nat$.
Item (3) is  the \CGR\  applied to indices higher than $1$, and Item  (2) is related to 
the  \CGR\  that involves the index $1$, e.g.,
$(1,1,1,1,2)$ is allowed while
$(2,1,1,1,2 )$ is not.

Let $d=4$. Then, $A=6$ and $B=7$.  
Listed below are examples of \cs s that satisfy the \CGR\ for interval $4$.
If  relevant to  Item (2),
the first entries are written in boldface:
\begin{equation} 
( \mathbf{ 6,5,5,5,6},0,2,5,6),\ (7,4,5,5,6,5,5),\ (\mathbf{6,6},3,5,5,0,4). 
\label{eq:cs-example}
\end{equation}

The Chung-Graham expansions are generalized as follows.
 
 \begin{theorem}
\label{thm:main} 
Let $d$ be a positive even integer.
Then, for each $n\in\nat$,
there is a unique \cs\ $\ep$ satisfying the Chung-Graham \roe\ 
for even interval $d$ 
 such that $n=\sum_{k=1}^\infty \ep_k F_{2+d(k-1)}$.
 
\end{theorem}

A majority of research works in the literature focus on finding a \roe\ given a sequence, and
 a relatively small number of works concern the converse of this task.
Namely, given a \roe, 
   what are the sequences by which each $n\in\nat$ is uniquely expanded   under the rule?
If such a sequence exists, we call it a {\it a fundamental (or base) sequence} for the \roe;
in this note, we shall call it a base sequence.
 This direction of research may begin by imposing some reasonable requirements 
 on the \roe.
 Introduced  in \cite{author} by the author is a general approach to the rule of expansion, and 
 if  a   rule of expansion satisfies the general principle,
  by \cite[Theorem 16]{author},  there is a unique increasing \funds\ for the \roe.
It follows from this theorem that for the \CGR\ for even interval $d$, 
  there is a unique increasing 
\funds, namely, $\set{F_{2+d(k-1)}}_{k=1}^\infty$.
For the \zec\ rule of expansion described in (\ref{eq:Z-rule}), Daykin proved in  \cite{daykin} that 
there is a unique \funds\ among all sequences in $\nat$, namely, 
the \fib\ sequence.
Arguably, this  is the most intriguing definition of the \fib\ sequence. 

We present two different proofs of Theorem \ref{thm:main}.
In Section \ref{sec:proof}, we prove the existence and uniqueness described in 
Theorem \ref{thm:main} using 
the greedy algorithm, which is the typical approach. 
 In Section \ref{sec:cs} and \ref{sec:expansions}, we prove the theorem using the \lex\ order 
 defined on the set of the \cs s.
 In the last two sections, the general approach introduced in \cite{author} are revisited in the context of the current paper.  The readers can use the approach introduced in these sections to find a \roe\ for the case of $d$ being odd or 
 the case of other linear recurrences.
 
Given a \cs\ $\ep$,
we denote by $\ord(\ep)$ the largest index $\ell$ such that $\ep_\ell\ne 0$ if $\ep$ is not the zero sequence.
 If $\ep$ is the zero sequence, then we define $\ord(\ep):=1$.
 If $\ell=\ord(\ep)$, then
 we also identify $\ep$ with the finite tuple
 $(\ep_1,\ep_2,\dots, \ep_\ell)$.

\section{The recurrence and the maximal \cs s} \label{sec:CG-expansions}

The \CGR\ for even interval $d$ is derived from a linear recurrence that $\set{F_{2+d(k-1)}}_{k=1}^\infty$ satisfies.
We discuss in this section the recurrence and its relationship with the  \roe.

Let 
$\phii=\tfrac12(1+\sqrt 5)$ be the golden ratio, and $\tphii$
be its Galois conjugate $\frac12(1-\sqrt 5)$.
By Binet's Formula,
\begin{equation}\label{eq:Binet}
F_k=\tfrac1{\sqrt 5}(\phii^k -\tphii^k),\quad
K_k =\phii^k + \tphii^k 
\end{equation} 
where $K $ is the sequence defined in Section \ref{sec:roe}.
Let $H$ be the sequence given by $H_k =F_{2 + d(k-1)}$.
Then, the formula in (\ref{eq:Binet}) implies 
$$H_k = a (\phii^d)^k + b (\tphii^d)^k
\quad
\text{where }
a,b \in \ratn[\phii].
$$
\renewcommand{\phitilde}{\widetilde{\phii}}
Notice that $\phii^d = \tfrac12( K_d + F_d \sqrt 5)$ for all $d\in\nat$, which
follows   trivially 
from (\ref{eq:Binet}).
Then, the minimal polynomial for $\phii^d$ is $x^2 - K_d x +\tfrac14(K_d^2 - 5 F_d^2)$.
\begin{lemma}
For $d\in\nat$, we have $K_d^2 - 5 F_d^2=(-1)^d \cdot 4$.
\end{lemma}
\begin{proof}
Notice that $ \phii \phitilde = -1$, and hence, $ \phii^d \phitilde^d =( -1)^d$.
The assertion  follows from 
\GGG{
\phii^d \phitilde^d =\tfrac12(K_d - F_d \sqrt 5)
\cdot \tfrac12 (K_d + F_d \sqrt 5) = \tfrac14( K_d^2 - 5 F_d^2 ).
}
\end{proof}

Thus, the two numbers $\phii^d$ and $\phitilde^d$ satisfy
the characteristic equation $x^2 = K_d x -(-1)^d$.
If $d$ is even, then the sequence $H$ satisfies the recurrence
\begin{equation}
\label{eq:the-recurrence}
H_{k+2} = K_d H_{k+1}-H_k.
\end{equation}

For the sequences satisfying the linear recurrence with positive constant coefficients,
the standard \roe\  and its \funds\ associated with this recurrence  
 are well-known; see \cite{hamlin,mw}.
For the linear recurrence with some negative  constant coefficients,
a \roe\  and its \funds\  are introduced by the author in \cite{author}. 
However, the \funds\ introduced in \cite{author} is not equal to $H$ for each $d\ge 4$.
For these cases, we tweak the rule, so that  under a new one,  $H$ is a \funds.
The \roe\ introduced in Definition \ref{def:rule}  is a variation of the rule introduced in
\cite{author}.

We conclude this section 
  explaining a key component of the rule of expansion introduced in \cite{author}.
It is called 
{\it the maximal \cs\ of order $n$} denoted by $\beta\lv n$, and 
it
satisfies  the property:
\begin{equation}
\label{eq:beta-n} 
 1+ \Cdot{\beta\lv n} H = H_{n+1}. 
\end{equation}
This  can be viewed as a generalization of
 consecutive $9$s in the last digits of a decimal expansion, e.g.,
 $1 +12345\mathbf{9999}=1234\mathbf{6}0000 $ 
 where the block of $9$s is carried over to the next digit.

 For the sequence $H$, we may use the recurrence (\ref{eq:the-recurrence}) to find $\beta\lv n$.
 Notice that if $m\ge 3$, then 
\begin{equation}
\label{eq:unfold-diff} 
H_m - H_{m-1} = (K_d-2) H_{m-1} +( H_{m-1}- H_{m-2}). 
\end{equation}
Then, for $n\ge 3$, we have
\begin{align} 
H_{n+1}&=K_d H_n - H_{n-1}  \notag
	=(K_d -1)H_n + (H_n - H_{n-1}) \\
	&=(K_d -1)H_n + (K_d-2) H_{n-1} +(H_{n-1}-H_{n-2}) \quad \text{by (\ref{eq:unfold-diff})} \notag \\
	&
	=  (K_d -1)H_n + (K_d-2) H_{n-1} +\cdots  + (K_d-2) H_{2} +(H_2-1)  \notag \\
	&
	=  (K_d -1)H_n + (K_d-2) H_{n-1} +\cdots  + (K_d-2) H_{2} +(H_2-2)+1 . \label{eq:unfolded}
\end{align}
Thus, for property (\ref{eq:beta-n}), we may define the entries of $\beta\lv n$ to be the
 coefficients of the expansion  (\ref{eq:unfolded}).
 \begin{deF}

\rm \label{def:beta-n}
Define  the \cs\   $\beta\lv n$ which is of order $n$ as follows:
\GGG{
\beta\lv 1=(B),\quad \beta\lv 2 = 
(B-1,A),\\
\intertext{and for every $n\ge 3$,}
\beta\lv n:= 
(B-1,A-1,A-1,\dots,A-1,A-1,A),
}
which contains  $n$ terms.
%For example, $\beta\lv 5 =(B-1, A-1,A-1,A-1,A)$.
 
\end{deF}

The standard \roe\ introduced in \cite{author} chooses  the value of $H_2$ to be equal to $K_d$, i.e., $H_2-2=K_d-2$,
so that the entries of $\beta\lv n$ described in (\ref{eq:unfolded}) have preperiodic structure and $H$ becomes its \funds.
However, for our sequence $H$,  we have $H_2= F_{2+d} \ne K_d$ if $d\ge 4$, and we may keep $H_2-2$ as the first entry 
of $\beta\lv n$, which will break the preperiodic structure at the first entry.

 \begin{lemma}\label{lem:carry-over} 
The identity (\ref{eq:beta-n}) is satisfied for all $n\ge 1$.
\end{lemma}
\begin{proof}
Recall that $H$ satisfies the recurrence in (\ref{eq:the-recurrence}).
If $n=1$, then $1+\Cdot{\beta\lv 1} H =1+B H_1 = 1+F_{2+d}-1 = H_2$.
If $n=2$, then 
\begin{align*} 
1+
\Cdot{\beta\lv 2} H &=1+ (B-1)H_1 + A H_2 = 1 + (H_2-2) + AH_2\\
&=-1 + (A+1)H_2=H_3. 
\end{align*}
For $n\ge 3$, the identity (\ref{eq:unfolded}) proves the assertion.
\end{proof}

\section{Proof of Theorem \ref{thm:main}} \label{sec:proof}

Let $\cF$ be the collection  of coefficient sequences  that satisfy the \CGR\ for even interval $d$.

\begin{lemma} \label{lem:cut-off}
Let $\ep=(\ep_1,\cdots ,\ep_\ell)\in\cF$.
Then, the \cs\ $\ep':=(\ep_1,\dots, \ep_m)$ and 
$\wt \ep=(\ep_1,\cdots ,\ep_\ell,a)$ are members of $\cF$ 
for each $1\le m<\ell$ and $0\le a\le A-1$.

\end{lemma}

\begin{proof}
Notice that if $\ep'$ does not satisfy an item of Definition \ref{def:rule}, then    $\ep$ does not satisfy the item, either.
Since $\ep\in\cF$, it follows that $\ep'$ must satisfy the items of the rule.
Since  $\wt\ep_{\ell+1}\le A-1$ and $\ell+1\ge 2$, the only item of the rule that concerns 
 the $(\ell+1)$th entry is the first item, which is clearly true.
\end{proof}

\subsection{Existence}
We use induction on $n\in\nat$ to prove the existence of $\ep\in\cF$ such that 
$n=\Cdot \ep H$.
\begin{lemma} \label{lem:max-a}
Let $n>B$ be an integer, and let $\ell$ be the maximal index such that 
$H_\ell \le n$.
Then, $\ell\ge 2$ and $(A+1)H_\ell >n$.
\end{lemma}
\begin{proof}
Since $n\ge B+1 = F_{2+d}=H_2$, we have $\ell\ge 2$.
By the maximality of $\ell$, we have 
\GGG{
n< H_{\ell+1} = (A+1)H_{\ell} - H_{\ell-1}<(A+1)H_{\ell}.
} 
\end{proof}

If $1\le n \le B$, then $n=\Cdot \ep H$ where $\ep_1=n$ and $\ep_k=0$
for all $k\ge 2$.  Since $\ep\in\cF$, this proves the existence
for $n\le B$.
Let $n\ge B+1$.
Let $\ell$ be the maximal index such that 
$H_\ell \le n $, and let $a$ be the largest positive integer $a $ such that $a H_\ell \le n$.
Then, $\ell\ge2$ and  $a\le A$ by Lemma \ref{lem:max-a}.

 Let $n':=n - aH_\ell$.
 By the induction hypothesis, $n'=\Cdot{ \ep'} H$ for some $\ep'\in\cF$.
The maximality of $a$ implies  $\ord(\ep')<\ell$.
 If $a\le A-1$, then $\ep:=(\ep'_1,\dots,\ep'_{\ell-1},a)\in \cF$ 
 by Lemma \ref{lem:cut-off} and $n=\Cdot \ep H$.
Suppose $a=A$.
    If $\ell=2$ and  $n'\ge B$, then Lemma \ref{lem:carry-over} implies $n=AH_2+n'\ge  AH_2+B=H_3$, 
    which contradicts the maximality of $\ell$.
    Thus, if $\ell=2$, then $n'<B$, $\ep:=(n',A)\in\cF$, and 
    $n=\Cdot \ep H$.
    
 Suppose $\ell\ge 3$, and let $m\ge 0$ be the maximal index  such that 
\begin{equation}\label{eq:m}
\wt n:=\sum_{k=\ell-m}^{\ell-1} (A-1) H_k + AH_\ell \le n.
\end{equation} 
Consider the case    $\ell-m\le 2$.
Then, $n-(\sum_{k=2}^{\ell-1} (A-1)H_k + AH_\ell)<B$ since
\GGG{
n-(\sum_{k=2}^{\ell-1} (A-1)H_k + AH_\ell)\ge B 
\implies
n\ge 1+\Cdot{\beta\lv \ell} H = H_{\ell+1},
}
which contradicts the maximality of $\ell$.
Thus, $n=\ep_1 H_1 +\sum_{k=2}^{\ell-1} (A-1)H_k + AH_\ell$ where $\ep_1<B$, which 
satisfies the \CGR\ for the interval $d$.

Let $\ell-m\ge 3$,
and let $b$ be the maximal coefficient such that $bH_{\ell-m-1}\le n - \wt n$.
   The maximality of $m$ implies $b \le A-2  $.
  By the induction hypothesis, $n-\wt n - bH_{\ell-m-1} = \Cdot \ep H$ for some $\ep\in\cF$.
The maximality of $b$ implies $q:=\ord(\ep)<\ell-m-1$. 
Thus, we have  the following expansion:
\GGG{
n=\sum_{k=1}^q \ep_k H_k + b H_{\ell-m-1} + \sum_{k=\ell-m}^{\ell-1} (A-1) H_k
+ A H_\ell,
}
  which satisfies the \CGR\ for the interval $d$.
This concludes the proof of the existence.

\subsection{Uniqueness}
Notice that 
\begin{equation} 
 A=K_d-1 \le B=F_{2+d}-1 \label{eq:A<B}
\end{equation}  for all $d\ge 1$, which can be proved by induction.
The equality in (\ref{eq:A<B}) holds only when $d=2$.

\begin{lemma} \label{lem:n<H}
Let $\ep\in\cF$ and $\ell:=\ord(\ep)$.
 Then, $\Cdot \ep H <H_{\ell+1}$.
\end{lemma}
\begin{proof}
We use induction on $\ell$.
If $\ell=1$, then $n\le B H_1 <B+1=H_2$.
Assume that there is $\ell \ge 1$ such that 
the statement is true for all $1\le \ell'\le \ell $.
Let $\ep\in \cF$ such that $\ord(\ep)=\ell+1$.
Suppose that  $\ep_{\ell+1}\le A-1$.
Then, by Lemma \ref{lem:cut-off}, $\ep'=(\ep_1,\dots, \ep_\ell)$ is a member of $\cF$.
By the induction hypothesis, 
  $\Cdot{ \ep'} H < H_{\ell+1}$, and hence,
  \GGG{
  \Cdot \ep H < H_{\ell+1} +\ep_{\ell+1} H_{\ell+1}
  \le A H_{\ell+1} < \Cdot {\beta \lv{\ell+1}}H <H_{\ell+2}.
  }
  
Suppose that  $\ep_{\ell+1}=A$, and let 
$m\ge 0$ be the maximal index such that $\ep_k=A-1$ for all $k$
in the interval $[\ell+1-m,\ell]$.
If $\ell+1-m\le 2$, then Item (2) of Definition \ref{def:rule} and 
(\ref{eq:A<B}) imply
$\Cdot \ep H\le \Cdot {\beta \lv{\ell+1}}H<H_{\ell+2}$.
Let $\ell+1-m\ge 3$.
Then, Item (3) of Definition \ref{def:rule} and the maximality of $m$ 
imply that $\ep_{\ell -m}\le A-2$.
Notice that $\ep':=(\ep_1,\dots, \ep_{\ell-m-1})\in\cF$ by Lemma \ref{lem:cut-off}.
By the induction hypothesis, 
$\Cdot {\ep'} H < H_{\ell-m}$, which implies
\AAA{
 \Cdot \ep H < H_{\ell-m} + \sum_{k=\ell-m}^\infty \ep_k H_k
  &\le \sum_{k=\ell-m}^\ell(A-1) H_k +A H_{\ell+1}\\
  &< \Cdot {\beta \lv{\ell+1}}H<H_{\ell+2}.
  }

\end{proof}

Let $\ep$ and $\delta$ be members of $\cF$ such that
\begin{equation} 
\Cdot \ep H = \Cdot \delta H.\label{eq:ep=delta}
\end{equation} 
Suppose that there is a largest index $q$ such that $\ep_q \ne \delta_q$, and 
without loss of generality, assume $\ep_q < \delta_q$.
Then,
\GGG{
H_q\le  \sum_{k=1}^{q-1} \delta_k H_k 
+(\delta_q - \ep_q) H_q
=\sum_{k=1}^{q-1} \ep_k H_k .
}
It follows from Lemmas \ref{lem:cut-off} and \ref{lem:n<H}
that
$
H_q\le \sum_{k=1}^{q-1} \ep_k H_k< H_q$,
which is a contradiction.
Thus,
we prove  $\ep_k=\delta_k$ for all $k\in\nat$. 

\section{Construction of coefficient sequences} \label{sec:cs}

In this section, we demonstrate the construction of the set $\cEE$ of \cs s that was introduced in 
\cite{author}.  Each positive integer is uniquely expanded by $H$ with \cs s in $\cEE$.
 Definition \ref{def:rule} is a concise description of the \cs s in $\cEE$.
In Section \ref{sec:expansions},
we shall provide a different proof of Theorem \ref{thm:main} utilizing the structure of the \cs s constructed in this section.

We begin by defining  a \lex\ order on \cs s.
\begin{deF}
\rm
Given two \cs s $\ep$ and $\tau$, we define
   $\ep < \tau$ if there is an index $\ell\ge 1$ such that 
   $\ep_\ell < \tau_\ell$ and $\ep_k=\tau_k$ for all $k>\ell$.
   This defines a total order on the set of \cs s, i.e., 
   given two \cs s $\ep$ and $\tau$, it is either 
   $\ep <\tau$, $\ep=\tau$, or $\tau <\ep$.
  
  Let $\cEE$ be a set of \cs s.
 For $\ep\in \cEE$, if there is a smallest \cs\ in $\cEE$ that is greater than $\ep$, 
then $\ep$ is called {\it the least upper bound of $\ep$ in $\cEE$}. 
It is denoted by $\lub_\cEE(\ep)$, and its $n$th iteration by $\lub_\cEE^n(\ep)$ if exists.
When $\cEE=\cE$, we simply denote it by $\lub$ instead of $\lub_\cE$.
\end{deF}

 We note here that the \cs\ $\beta\lv n$ is indeed the largest one  
 in $\cF$
  among the \cs s of order less than or equal to $n$.

\begin{lemma} \label{lem:maximal}
Let $n\in\nat$.
If $\ep\in\cF$ and $\ord(\ep)\le n$, then
  $\ep \le \beta\lv n$ (with respect to the \lex\ order).
\end{lemma}

\begin{proof}
If $n\le 2$, the case is trivial, and we leave it to the reader.
Suppose that $n\ge 3$ and there is $\ep\in\cF$ such that $\ord(\ep)=n$ and 
$\ep >\beta\lv n$.
Clearly, $\ep_n=A$ must be the case.
If  there is $2\le k\le n-1$ such that $ \ep_k>A-1$,
then $\ep_k=A$ must be the case, which violates Item (3) of the \CGR\ for
even interval $d$.
Thus,
$\ep_k=A-1$  is true for all $2\le k\le n-1$.
Then, it follows $\ep_1>B-1$, which violates Item (2) of the rule.
Therefore,   no \cs\  of order $n$ in $\cF$   is greater 
than $\beta\lv n$.
\end{proof}

The main idea of the construction of $\cEE$ is declaring \cs s $\ep<\beta\lv \ell$ given $\ell\in\nat$,
so that $\lub_\cEE^n(\ep)$ gradually build up to $\beta\lv \ell$ carrying over entries as in (\ref{eq:beta-n}).
These smaller \cs s are constructed with \cs s called {\it proper blocks}.
\begin{deF}
\rm
\label{def:proper}
Coefficient sequences $\zeta$  and $\zeta'$ are called a {\it upper proper block} and a {\it lower proper block}, respectively,
 if 
 \begin{align*}
\zeta &=(\zeta_1,\beta\lv n_3,\beta\lv n_4,\dots, \beta\lv n_n),\\
\zeta' &=(\zeta_1',\beta\lv m_2,\beta\lv m_3,\dots, \beta\lv m_m),
\end{align*}  
and $\zeta_1<\beta\lv n_2$ for some $n\ge 2$ and $\zeta_1'<\beta\lv m_1$ for some $m\ge 1$.

\end{deF}
Notice that $A=K_d-1 < B=F_{2+d}-1$  for all $d\ge 1$.
The upper proper blocks of order $1$, which are obtained with $n=2$, are $(\zeta_1)$ where $0\le \zeta_1\le A-1$,
and the lower proper blocks of order $1$, which are obtained with $m=1$, are $(\zeta_1')$ where $0\le \zeta_1'\le B-2$.
Thus, the upper proper blocks of order $1$ can be considered as lower proper blocks of order $1$ as well.
In particular, $(0)$ is an upper and lower proper block.
Listed below are the proper blocks of order up to $4$:
\begin{align*} 
\text{Upper proper blocks: }&
(c,A),\ (c, A-1, A),\ (c,A-1,A-1,A),\quad c\le A-2,\\
\text{Lower proper blocks: }&
(c',A),\ (c',A-1,A),\ (c',A-1,A-1,A),\quad c'\le B-2. 
\end{align*}
In general, the lower proper blocks are similar to the upper proper blocks, but 
the first entries of the lower proper blocks can go up to a value higher than
those of the upper proper blocks. 
This implies that all upper proper blocks can be considered as lower proper blocks.

\begin{deF}
\rm
Let $\cFo$ denote the  collection of \cs s obtained by concatenating a single lower proper block with multiple 
upper 
proper blocks to the right side.
We use the symbol $\vee$ for the concatenation operator, i.e.,
if $\ord(\ep)=\ell$, then
$$\ep\vee \tau :=(\ep_1,\dots,\ep_\ell,\tau_{1 },\tau_{2 },\dots).$$
Then, all \cs s $\ep\in\cFo$ are in the form of 
$$\zeta\lv 0 \vee \zeta\lv 1\vee \cdots \vee \zeta\lv m$$
where $\zeta\lv 0$ is a lower proper block and $\zeta\lv k$ for $1\le k\le m$ are 
upper proper blocks.
\end{deF}

Listed below are some examples where $d=4$, for which $A=6$ and $B=7$:
\begin{gather} 
(5,5,6)\vee(0,5,6) \vee(0)\vee(2)\vee (5)=(5,5,6,0,5,6,0,2,5)\in\cFo,\notag\\
(6)\vee(4,5,6)\vee(5,5,6)  = (6,4,5,6,5,5,6)\not\in\cFo.\notag
\end{gather} 
Notice that $(5,5,6)$ and $(6)$  are lower proper blocks that are not  upper proper blocks. 
Since $(5,5,6)$ is not an upper proper block, the second \cs\ is not a member of $\cFo$.

The \cs s   in $\cFo$ are constructed  to allow for an increase in the first entry while remaining a \cs\ that is less than or equal to some $\beta\lv \ell$.  For example, if $d=4$, we have
\begin{gather} 
\ep=(0,5,6)\vee(0,5,6)\overset{\lub^5}{\longrightarrow}
(5,5,6)\vee(0,5,6)\in\cFo\HSW{.3}\notag\\
\iffalse
\overset{\lub}{\longrightarrow}
(2,5,6)\vee(0,5,6)\overset{\lub^3}{\longrightarrow}
(5,5,6)\vee(0,5,6)\in\cFo 
\fi  
\HSW{.2}
\overset{\lub }{\longrightarrow}(6,5,6)\vee(0,5,6)\not\in\cFo,\quad
 \beta\lv 3=(6,5,6).  \label{eq:656}
\end{gather}
By allowing $\beta\lv n$ as a first block, we complete the construction.
We refer to an upper or lower proper block or $\beta\lv n$ simply as a {\it block}.
\iffalse
Since we have a maximal \cs\ in $\tau$, by (\ref{eq:beta-n}), we declare that it is carried over to the next block $(0,5,6)$, and 
 it is crucial that the next block is an upper proper block, so that it can grow to 
 another maximal \cs, i.e.,
 \GGG{
 (6,5,6)\vee(0,5,6)\overset{\lub }{\longrightarrow}
  (0)\vee(0)\vee(0)\vee (1,5,6)\HSW{.2}\\
  \HSW{.2}
  \overset{\lub^{H_4-1} }{\longrightarrow}
   (6,5,6)\vee (1,5,6)
\overset{\lub }{\longrightarrow} (0)\vee(0)\vee(0)\vee (2,5,6)\\
\HSW{.3}
\overset{\lub }{\longrightarrow}\cdots
 \overset{\lub }{\longrightarrow}
 (0)\vee(0)\vee(0,5,5,6).
  }
 \fi

\begin{deF}
\rm
Let $\cEE$ be the collection  of the 
\cs s that are in the form of either $\beta\lv n\vee \mu$  or $\tau$ 
where $n\ge 1$, $\mu$ is the concatenation of upper proper blocks,  and $\tau\in\cFo$.
\end{deF}

For both cases  $\ep= \beta\lv n\vee \mu$  and $\ep=\tau$ considered above,
if  $\ord(\ep)=\ell$, then
  a block $(\ep_{\ell-s},\ep_{\ell-s+1},\dots,\ep_\ell)$ is uniquely determined.
  By induction, it follows that 
  the decomposition of $\mu$ and $\tau$ into   proper blocks is uniquely determined.
  We call the following {\it the decomposition of $\ep$ into   blocks}:
  $$ \ep=\beta\lv n\vee \zeta\lv 1 \vee \cdots \vee \zeta\lv m,
  \ \text{ or } \ep=  \zeta\lv 1 \vee \cdots \vee \zeta\lv m $$
  where $\zeta\lv t$ are (upper or lower) proper blocks. 
  
\begin{prop}
For each even integer $d\ge 2$, we have $\cF = \cEE$.
\end{prop}

\begin{proof}
Let $\ep\in \cEE$.  Then, it has  a decomposition of $\ep$ into   blocks.
The maximal \cs s and proper blocks satisfy Item (1) of the \CGR\ for  $d$, and  so does $\ep$.
Item (2) refers to the case that the decomposition has a lower proper block of order greater than $1$, 
and hence, the item is satisfied.
Notice that the entries of the upper proper blocks are less than or equal to $A$, and 
if $\ep\in\cEE$ and $\ep_m = A$ for $m\ge 2$, then the decomposition of $\ep$ implies that 
$(\ep_{m-j},\dots, \ep_m)$ for some $j>0$ is a block.
Thus,
Item (3) refers to the case that there are two or more  blocks $\zeta$ in the decomposition 
such that
$\zeta_\ell =A$ where $\ell=\ord(\zeta)\ge 2$.
Thus, the  block appearing to the right side of the other one must be an upper proper block $\xi$.
By the definition of upper proper blocks, $\xi_1< A-1$, and hence, Item (3) is satisfied.
This proves $\ep\in\cF$.

For the inclusion $\cF \subset \cEE$, we use  induction on $\ord(\ep)$.
Let $\ep\in\cF$, and let $n:=\ord(\ep)$.   
The  cases of   $n\in\set{1,2}$ are a straightforward exercise to verify that they have a decomposition into   blocks,
and we leave it to the reader.
Suppose that  $\ep\in\cF$ and $\ord(\ep)=n\ge 3$.
If $\ep_n<A$, then $(\ep_1,\dots,\ep_{n-1})$  satisfies the \CGR\ for $d$, and by 
the induction hypothesis, $\ep$ has a decomposition into   blocks since $(\ep_n)$ is an upper 
proper block.
If $\ep_n=A$ and $\ep_{n-1}\le A-2$, then 
   $(\ep_1,\dots,\ep_{n-2})$ satisfies the \CGR\ for $d$.
   Since $(\ep_{n-1},\ep_n)$ is an upper proper block, this case is proved  by the induction hypothesis.
Suppose that $\ep_n=A$ and there is a smallest index $\ell$ such that  $\ep_k= A-1$ for all $\ell \le k\le n-1$.
If $\ell=1$, then $A-1<B-1$ implies that $\ep$ is a lower proper block.
If $\ell=2$, then Item (2) of the \CGR\ for $d$ implies $\ep_1<B$, and hence, $\ep$ is a lower proper block or 
$\beta\lv n$.
If $\ell\ge 3$, then Item (3) of the \CGR\ for $d$ implies that $\ep_{\ell-1}\ne A$, and hence, $\ep_{\ell-1}<A-1$ must be the case.
Then, $(\ep_1,\dots,\ep_{\ell-2})$ still satisfies the \CGR\ for $d$.
Since $(\ep_{\ell-1},\dots, \ep_n)$ is an upper proper
block, by the induction hypothesis,  $\ep\in\cEE$.
\end{proof}

\iffalse
The key principle to constructing the set $\cEE$, which is introduced in \cite{author},  is declaring when entries are carried over as in 
(\ref{eq:beta-n}).
We declare that 
$$\lub_\cEE(\ep) 
=\begin{cases}
(1+\ep_1,\ep_2,\ep_3,\dots) &\text{ if $(\ep_1,\dots,\ep_\ell)\ne \beta\lv \ell$
for all $\ell\in\nat$,}\\
(a_1,\dots,a_\ell,1+\ep_{\ell+1},\ep_{\ell+2},\dots )
& \text{ if $(\ep_1,\dots,\ep_\ell)= \beta\lv \ell$
for some $\ell\in\nat$ }
\end{cases}
$$
where $a_1=\cdots=a_\ell=0$.
\fi

\section{Expansions of least upper bounds} \label{sec:expansions}

  Notice that each entry of $\ep\in\cF$ is bounded by $B$.
   This implies that  for each $\ep\in\cF$, there are only finitely many $\tau\in\cF$ such that 
   $\tau<\ep$.
   Thus, under this order, $\cF$ satisfies the well-ordering principle.
It follows that $\lub(\ep)$ exists  for each $\ep\in \cF$.
Thus, $\cF = \set{ \lub^n(0) :  n\in \nat_0}$.
Define the function $\eval : \cF \to \nat_0$ given by 
$\ep\mapsto \Cdot \ep H$. 
In Theorem \ref{thm:increasing} below, we prove that this function is bijective.

Let us demonstrate how to find $\lub(\ep)$.
Notice that if $\tau= \zeta\lv 1 \vee \cdots \vee \zeta\lv m \in\cFo$ and $\ell=\ord(\zeta\lv 1)$, then 
$(1+\zeta\lv 1_1,\zeta\lv 1_2,\dots,\zeta\lv 1_\ell)$ remains as a block.
Thus, 
\begin{equation}
 \lub(\tau) = (1+\tau_1,\tau_2,\dots).\label{eq:carried-over-0}
\end{equation} 
 If
\begin{equation} \label{eq:carry-over}
 \ep = \beta\lv n\vee \zeta\lv 1 \vee \cdots \vee \zeta\lv m\in\cF 
\end{equation} 
where $\zeta\lv t$ are upper proper blocks, then we claim
\begin{equation}
\label{eq:carried-over} 
\lub(\ep) = (a_1,\dots, a_n, 1+\ep_{n+1},\ep_{n+2},\ep_{n+3},\dots) 
\end{equation}
where $a_k=0$ for all $k=1,\dots,n$.
For example, if $\ep$ is the \cs\ considered in (\ref{eq:656}) where $A=6$ and $B=7$,
then $\beta\lv 3=(6,5,6)$ implies
$$\ep=(\mathbf{6,5,6},0,5,6)\ \implies\lub(\ep) = (0,0,0,1,5,6).$$
We use this example to explain  (\ref{eq:carried-over}).
First of all, $(0,0,0,1,5,6)$ is an upper bound of $\ep$.
Suppose that $ \tau:=(a,b,c, 0,5,6)\in \cF$ and  $\ep<\tau$.
Then, comparing the   block decompositions of $\ep$ and $\tau$, we conclude that $(a,b,c)\in \cF$ and 
$(a,b,c)>\beta\lv 3$.  However, by Lemma \ref{lem:maximal}, $\beta\lv 3$ is the largest one with order $3$ in $\cF$, which
 proves that $\lub(\ep)$ is as described above.
Notice also that it is crucial that $\zeta\lv 1$ in (\ref{eq:carry-over}) is an upper proper block, 
so that $\lub(\ep)$ described in (\ref{eq:carried-over}) remains in $\cF$.

Let us prove Theorem \ref{thm:main}.
First, we relate  the expansion $\Cdot \ep H$ to the descriptions  of $\lub(\ep)$ given in
(\ref{eq:carried-over-0}) and (\ref{eq:carried-over}).
If $ \tau\in\cFo$, then  (\ref{eq:carried-over-0}) implies 
\begin{equation}\label{eq:tau}
1 + \Cdot \tau H =\Cdot{\lub(\tau)} H.
\end{equation}

We use Lemma \ref{lem:carry-over}   for the case $\ep=\beta\lv n \vee \tau$ where 
$\tau\in\cFo$.   
Suppose that $\ep$ is the \cs\ considered in (\ref{eq:carry-over}). 
Then, (\ref{eq:carried-over}) and  Lemma \ref{lem:carry-over} imply 
\begin{align}  \label{eq:ep}
1 + \Cdot \ep H = (1+\ep_{n+1}) H_{n+1} + \sum_{k=n+2}^\infty \ep_k H_k
=\Cdot{\lub(\ep)} H.
\end{align} 

Our main result Theorem \ref{thm:main}  is equivalent to the bijectivity of 
$\eval : \cF \to \nat_0$ given by $\ep \mapsto \Cdot \ep H$, which is the corollary of the following theorem.
\begin{theorem} \label{thm:increasing}
For $n\in\nat_0$, we have $\eval(\lub^n(0)) = n$.
\end{theorem}
\begin{proof}
By (\ref{eq:tau}) and (\ref{eq:ep}),
we have   
$$1 + \Cdot \ep H = \Cdot {\lub(\ep)} H$$ for all $\ep\in\cF$, i.e, 
$
 1 + \eval(\ep) = \eval(\lub(\ep)) \notag%\label{eq:eval+1}
$.
By induction, this proves that 
$\eval(\lub^n(0)) = n$ for all $n\in\nat_0$.  
\end{proof}

\section*{Acknowledgement}
The author would like to thank the referee for carefully reading the paper and   recommending to consider a direct proof of existence and uniqueness.

\end{document}

%% file: basic-layout.tex
\newcommand\VS[1]{\mbox{\rule{0pt}{#1}}}

%% file: MathMacro.tex
\renewcommand{\implies}{\Rightarrow}

\newcommand{\nat}{\mathbb{N}}
\newcommand{\ratn}{\mathbb{Q}}

\newcommand{\ncom}{\newcommand}

\def\abs#1{\left\vert #1 \right\vert}

\def\set#1{\lbrace #1 \rbrace}

\newcommand{\al}{\alpha}

\newcommand{\ep}{\epsilon}
 \ncom{\Del}{\Delta}

\newcommand{\GGG}[1]{
\begin{gather*}
#1
\end{gather*}
}
\newcommand{\AAA}[1]{
\begin{align*}
#1
\end{align*}
}

%% file: algebra.tex
\newcommand{\ord}{\operatorname{ord}}

%% file: Zec-local.tex
\newcommand{\zec}{Zeckendorf}

\newcommand{\fib}{Fibonacci}
\newcommand{\funds}{base  sequence}

\newcommand{\HSW}[1]{\rule{#1\linewidth}{0pt}}

\newcommand{\cE}{\mathcal{F}}
\newcommand{\cEE}{\mathcal{E}}

\newcommand{\lex}{lexicographical}

\newcommand{\cs}{coefficient sequence}

 \newcommand{\cF}{\mathcal{F}}
 \newcommand{\cFo}{\cEE_\circ}

\newcommand{\seq}[1]{\set{#1_k}_{k=1}^\infty}

\newcommand{\wt}{\widetilde}
\newcommand{\phitilde}{\widetilde{\phi}}

\renewcommand{\implies}{\ \Rightarrow\ }

\newcommand{\Cdot}[2]{ \sum_{k=1}^\infty #1_k #2_k }
\newcommand{\eval}{\mathrm{eval}}
\newcommand{\lv}[1]{^{(#1)}}
\newcommand{\lub}{\mathrm{lub}}
\newcommand{\phii}{\varphi}
\newcommand{\tphii}{\widetilde\varphi}
\newcommand{\roe}{rule of expansion}
\newcommand{\CGR}{Chung-Graham \roe}